%% file: Kurepa.tex
\documentclass[a4paper,12pt]{article}

\usepackage{amscd}
\usepackage{amsfonts}
\usepackage{amsmath}
\usepackage{amssymb}
\usepackage{amsthm}
\usepackage{color} % \textcolor
\usepackage[T1]{fontenc} % changing encode
\usepackage{here} % [h] for tables
\usepackage{mathrsfs} % \mathscr
\usepackage{txfonts} % colneqq
\usepackage[all]{xy} % xypic
\usepackage{algorithm} % pseudocode
\usepackage{algpseudocode} %pseudocode
\usepackage{diagbox} % \diagbox

\allowdisplaybreaks

\theoremstyle{plain}
\newtheorem{thm}{Theorem}[section]

\newtheorem{prp}[thm]{Proposition}
\newtheorem{crl}[thm]{Corollary}

\theoremstyle{definition}

\newcommand{\vs}[1][0.2]{\vspace{#1in}\noindent\ignorespaces}
\newcommand{\ba}{\begin{array*}}
\newcommand{\ea}{\end{array*}}
\newcommand{\be}{\begin{eqnarray*}}
\newcommand{\ee}{\end{eqnarray*}}
\newcommand{\bi}{\begin{itemize}}
\newcommand{\ei}{\end{itemize}}
\newcommand{\bb}{\vs\begin{itembox}}
\newcommand{\eb}{\end{itembox}}
\newcommand{\bc}{\begin{center}}
\newcommand{\ec}{\end{center}}
\newcommand{\bs}{\vs\begin{screen}}
\newcommand{\es}{\end{screen}}

\def\ens#1{{\mathchoice{\left\{ #1 \right\}}{\{ #1 \}}{\{ #1 \}}{\{ #1 \}}}}
\def\set#1#2{{\mathchoice{\left\{ #1 \ \middle| \ #2 \right\}}{\{ #1 \mid #2 \}}{\{ #1 \mid #2 \}}{\{ #1 \mid #2 \}}}}
\def\r#1{\text{\rm #1}}

\newcommand{\bN}{\mathbb{N}}

\newcommand{\bP}{\mathbb{P}}
\newcommand{\bQ}{\mathbb{Q}}
\newcommand{\bR}{\mathbb{R}}

\newcommand{\bZ}{\mathbb{Z}}

\newcommand{\cA}{\mathscr{A}}

\newcommand{\N}{\bN}
\newcommand{\Q}{\bQ}
\newcommand{\R}{\bR}
\newcommand{\Z}{\bZ}

\newcommand{\Fp}{\mathbb{F}_p}

\algnewcommand\algorithmicbreak{{\bf break}}
\algnewcommand\Break{\algorithmicbreak{}}
\algnewcommand\algorithmiccontinue{{\bf continue}}
\algnewcommand\Continue{\algorithmiccontinue{}}

\title{Parallel Searching Method for Verifying Finiteness of Prime Numbers Presented as Specific Linear Combinations of Factorials}
\author{Tomoki Mihara}
\date{}

\begin{document}

\maketitle
%\address
\input{Abstract}
\tableofcontents
%\fn{}

\input{Introduction}
\input{Convention}
\input{ProblemStatement}
\input{GroupingMethod}
\input{Application}

\input{References}

\end{document}

%% file: Abstract.tex
\begin{abstract}
As a new approach to search a counterexample of variants of Kurepa's conjecture, we investigate a parallel searching method of prime numbers $p$ satisfying
\begin{eqnarray*}
\sum_{k=0}^{p-1} c_k k! \equiv 0 \pmod{p}
\end{eqnarray*}
for a given integer sequence $(c_k)_{k=0}^{\infty}$. As an application, we prove that there are only finitely may prime numbers of the form $\sum_{k=0}^{n} (1+dk)!$ with $(d,n) \in (\mathbb{N}_{\leq 25} \setminus \{0,8,18,20,23,25\}) \times \mathbb{N}$.
\end{abstract}

%% file: Introduction.tex
\section{Introduction}
\label{Introduction}

Let $\bP$ denote the set of prime numbers. \DJ.\ Kurepa defined the left factorial
\be
!n \coloneqq \sum_{k=0}^{n-1} k!
\ee
for $n \in \N$, and proposed a conjecture that the equality
\be
\gcd(!n,n!) = 2
\ee
holds for any $n \in \N_{> 1}$, or equivalently, there is no solution $p \in \bP_{> 2}$ of the congruence relation
\be
!p \equiv 0 \pmod{p},
\ee
in \cite{Kur71}. Although the conjecture is still open, there have ever been many attempts to find a counterexample, i.e.\ a solution of the congruence relation. For example, V.\ Andrejei\'c and M.\ Tatarevic searched in \cite{AT16} a solution approximately $240$ days, and verified that there is no solution below $2^{34}$. There are also studies on solutions of its variants. For example, $\check{\r{Z}}$.\ Mijajlovi\'c searched in \cite{Mij99} a solution $p \in \bP$ of another congruence relation
\be
\sum_{k=1}^{p-1} (-1)^{k-1} k! \equiv 0 \pmod{p}
\ee
and found the least solution $p = 3612703$, in order to show that there are only finitely many prime numbers of the form $\sum_{k=1}^{n} (-1)^{k-1} k!$ with $n \in \N_{> 0}$. It is still open whether there are infinitely many $p \in \bP$ satisfying another congruence relation
\be
\sum_{k=0}^{p-1} (-1)^{k-1} k! \equiv 0 \pmod{p}
\ee
or not, where the reader should be careful that $k$ starts from $0$, and currently only five solutions $p = 2, 5, 13, 37, 463$ are found (cf.\ \cite{SS09} \S 5). We note that if there are infinitely many $p \in \bP$ not satisfying the congruence relation above, i.e.\ satisfying another congruence relation
\be
\sum_{k=0}^{p-1} (-1)^{k-1} k! \not\equiv 0 \pmod{p},
\ee
then its image in the ring
\be
\cA \coloneqq \prod_{p \in \bP} \Fp \bigg/ \bigoplus_{p \in \bP} \Fp
\ee
of integers modulo infinitely large prime numbers is not $0$. In number theory, the ring $\cA$ plays a role of the base ring in the study of finite multiple zeta values, which are finite analogues of multiple zeta values introduced by D.\ Zagier. Starting from \cite{KZ}, there have been various studies on specific elements of $\cA$ by number theorists. The image of the value above coincides with the image $e_{\cA}$ of $(\sum_{k=0}^{p-1} 1/k! + p \Z)_{p \in \bP}$ by Wilson's theorem (cf.\ \cite{SS09} Lemma 5.2), and $e_{\cA}$ is studied as a finite analogue of Napier's constant $e$ (cf.\ \cite{MMY26} Theorem 1.1). In this sense, study of distribution of solutions of the congruence relation above is closely related to the value of $e_{\cA}$, which is currently not known to be transcendental over $\Q$, irrational, or even non-zero.

\vs
We consider a generalisation of these problems. Let $(c_k)_{k=0}^{\infty}$ be a sequence of integers (typically taking values in $\ens{-1,0,1}$). Is there a solution $p \in \bP$ of the following congruence relation?
\be
\sum_{k=0}^{p-1} c_k k! \equiv 0 \pmod{p}
\ee
If the answer is yes, then there are only finitely many prime numbers of the form $\sum_{k=0}^{n} c_k k!$ with $n \in \N_{> 0}$, because the least solution $p$ divides $\sum_{k=0}^{n} c_k k!$ for any $n \in \N_{\geq p-1}$ (cf.\ Proposition \ref{finiteness}). Therefore searching a solution for the congruence relation is within the realm of the study of infiniteness of prime numbers of the given form.

\vs
For $n \in \N_{> 0}$, the brute-forth searching of a solution $p$ not greater than $n$ using modular arithmetic modulo $p$ requires $\Theta(n \pi(n) (t_{\r{m}}(n) + t_{\r{r}}(n)))$ time complexity and $\Theta(n s_{\r{m}}(n) + s_{\r{r}}(n))$ space complexity in cases with no solution, where $\pi(n)$ denotes the number of prime numbers not greater than $n$, $t_{\r{m}}(n)$ (resp.\ $s_{\r{m}}(n)$) denotes the worst case time (resp.\ space) complexity of $x + y$, $x - y$, $x \times y$ with $(x,y) \in \N_{\leq n}^2$, and $t_{\r{r}}(n)$ (resp.\ $s_{\r{r}}(n)$) denotes the worst case time (resp.\ space) complexity of $\lfloor x / y \rfloor$, $x \bmod y$ with $(x,y) \in \N_{\leq n^2} \times (\N_{\leq n} \setminus \ens{0})$ with respect to fixed implementations. The most standard way to effectively search a solution among preceding studies is concurrent/parallel computation of the brute-forth searching (cf.\ \cite{Mij90}, \cite{Mij99}, \cite{AT16}), but the approach restricts programming languages to be ones with concurrent/parallel feature.

\vs
In our study, we investigate another parallel searching method based on grouping of prime numbers, which does not restrict a programming language. Let $B \in \N_{> 0}$ be a fixed constant, typically chosen as $10^4$, indicating the group size. For $n \in \N_{> 0}$, the grouping method requires $\Theta((n \log \log n) t_{\r{m}}(n) + \pi(n)((1 + \frac{n}{B}) t_{\r{m}}(n^B) + \frac{1}{B} t_{\r{r}}(\max \ens{n,B}) + \frac{n}{B^2} t_{\r{r}}(n^B)))$ time complexity and $\Theta(n s_{\r{m}}(n) + s_{\r{m}}(n^B) + s_{\r{r}}(n^B))$ space complexity in cases with no solution. In particular, this method heavily depends on the implementation of big integer arithmetic.

\vs
Since $t_{\r{m}}(n^B)$ and $t_{\r{r}}(n^B)$ asymptotically dominate $B t_{\r{m}}(n)$ and $B t_{\r{r}}(n)$ respectively in standard implementations, the grouping method might look less effective than the brute-forth method. However, as far as we experimented with $n \leq 3 \times 10^7$ and $B \in \ens{10^3,10^4}$ using Python, the asymptotic comparison failed, i.e.\ the grouping method performed much better than the brute-forth method. In the case where constant factors in computational complexity, e.g.\ function calls, cannot be ignored, the divergence of $t_{\r{m}}(n^B)$ and $t_{\r{r}}(n^B)$ relative to $B t_{\r{m}}(n)$ and $B t_{\r{r}}(n)$ does not occur. Therefore, the high efficiency of the grouping method is not surprising. As a application of the grouping method, we obtain the following:

\begin{thm}
\label{main}
There are only finitely many prime numbers of the form
\be
\sum_{k=0}^{n} (kd+1)!
\ee
with $(d,n) \in (\N_{\leq 25} \setminus \ens{0,8,18,20,23,25}) \times \N$.
\end{thm}

We briefly summarise contents of this paper. In \S \ref{Convention}, we explain convention in this paper. In \S \ref{Problem Statement}, we explain the congruence relation in question. In \S \ref{Grouping Method}, we explain the grouping method and implementations. In \S \ref{Application}, we exhibits the results of the grouping method applied to specific problems, and prove Theorem \ref{main}.

%% file: Convention.tex
\section{Convention}
\label{Convention}

We denote by $\N$ the set of non-negative integers, and by $\bP$ the set of prime numbers.

\vs
For sets $X$ and $Y$, we denote by $X^Y$ the set of maps $Y \to X$. For $X \subset \R$ and $x \in X$, we set $X_{< x} \coloneqq \set{x' \in X}{x' < x}$ and $X_{\leq x} \coloneqq \set{x' \in X}{x' \leq x}$. We note that every $d \in \N$ is identical to $\N_{< d}$, and hence for a set $X$, $X^d$ formally means $X^{\N_{< d}}$, which is naturally identified with the set of $d$-tuples in $X$. For a set $I$, we denote by $\# I$ its cardinality.

\vs
For $n_{> 0} \in \N$, $\pi(n)$ denotes the number of prime numbers not greater than $n$, $t_{\r{m}}(n)$ (resp.\ $s_{\r{m}}(n)$) denotes the worst case time (resp.\ space) complexity of $x + y$, $x - y$, $x \times y$ with $(x,y) \in \N_{\leq n}^2$, and $t_{\r{r}}(n)$ (resp.\ $s_{\r{r}}(n)$) denotes the worst case time (resp.\ space) complexity of $\lfloor x / y \rfloor$ and $x \bmod y$ with $(x,y) \in \N_{\leq n^2} \times (\N_{\leq n} \setminus \ens{0})$ with respect to fixed implementations.

%% file: ProblemStatement.tex
\section{Problem Statement}
\label{Problem Statement}

In this section, we explain the problem statement of our study. Let $(c_k)_{k=0}^{\infty} \in \Z^{\N}$. We consider the following two simple questions:
\bi
\item[(1)] Are there infinitely many prime numbers of the form
\be
\sum_{k=0}^{n} c_k k!
\ee
with $n \in \N$?
\item[(2)] Is there a solution $p$ of the congruence relation
\be
\sum_{k=0}^{p-1} c_k k! \equiv 0 \pmod{p}
\ee
on $p \in \bP$?
\ei
The motivation of (1) comes from one of main topics in number theory studying sequences generating infinitely many prime numbers as an extension of Euclid's theorem on infiniteness of prime numbers. The motivation of (2) comes from (1) by the following:

\begin{prp}
\label{finiteness}
If there exists some solution $p \in \bP$ of the congruence relation in (2), there are at most 
\be
\# \left( \set{\sum_{k=0}^{n} c_k k!}{n \in \N_{< p}} \cup \ens{p} \right)
\ee
prime numbers of the form in (1). In particular, there are only finitely many prime numbers of the form in (1).
\end{prp}

The proof is essentially in \cite{Mij90}, although we are considering a generalised setting.

\begin{proof}
For any $n \in \N_{\geq p}$, we have
\be
\sum_{k=0}^{n} c_k k! \equiv \sum_{k=0}^{p-1} c_k k! + p! \sum_{k=p}^{n} c_k \prod_{j=p+1}^{k} j \equiv \sum_{k=0}^{p-1} c_k k! \equiv 0 \pmod{p},
\ee
and hence the left hand side is a prime number if and only if it coincides with $p$. This implies that the number
\be
\sum_{k=0}^{n} c_k k!
\ee
with $n \in \N$ is a prime number only if $n < p$ or it coincides with $p$.
\end{proof}

A standard technique to estimate the distribution of solutions of the congruence relation in (2) is a heuristic method under an intuitive assumption of randomness of the value
\be
\sum_{k=0}^{p-1} c_k k! \bmod p \in \N_{< p}
\ee
for varying $p \in \bP$ (cf.\ \cite{Mij99} p.\ 407, \cite{SS09} p.\ 6, \cite{AT16} \S 4), which yields expectation that there are infinitely many solutions. Of course, the randomness fails in trivial cases where $c_k = 0$ for all but finitely many $k \in \N$. Further, the randomness fails when we choose $(c_k)_{k=0}^{\infty}$ even non-trivially but artificially:

\begin{prp}
\label{emptiness}
There exists some computable sequence $(c_k)_{k=0}^{\infty} \in \Z^{\N}$ such that $c_0 = 0$, $c_k \in \ens{-1,1}$ for any $k \in \N_{> 0}$, and there is no solution of the congruence relation in (2).
\end{prp}

\begin{proof}
We define $(c_k)_{k=0}^{\infty}$ in the following recursive way:
\be
c_k \coloneqq
\left\{
\begin{array}{ll}
0 & (k = 0) \\
-1 & \left( k > 0 \land \left( \sum_{j=0}^{k-1} c_j j! \right) + k! \equiv 0 \pmod{k+1} \right) \\
1 & \left( k > 0 \land \left( \sum_{j=0}^{k-1} c_j j! \right) + k! \not\equiv 0 \pmod{k+1} \right)
\end{array}
\right.
\ee
We have $c_0 = 0$ by definition, $c_1 = 1$ by $c_0 0! + 1! = 1 \not\equiv 0 \pmod{2}$, and $c_2 = -1$ by $c_0 0! + c_1 1! + 2! = 3 \equiv 0 \pmod{3}$. Let $p \in \bP$. If $p = 2$, then we have
\be
\sum_{k=0}^{p-1} c_k k! = c_0 0! + c_1 1! + c_2 2! = -1 \not\equiv 0 \pmod{p}.
\ee
Suppose $p > 2$. First, if
\be
\left( \sum_{k=0}^{p-2} c_k k! \right) + (p-1)! \equiv 0 \pmod{p},
\ee
then we have $c_{p-1} = -1$, and hence
\be
\sum_{k=0}^{p-1} c_k k! = \left( \sum_{k=0}^{p-2} c_k k! \right) - (p-1)! \equiv -2 \times (p-1)! \not\equiv 0 \pmod{p}.
\ee
Next, if
\be
\left( \sum_{k=0}^{p-2} c_k k! \right) + (p-1)! \not\equiv 0 \pmod{p},
\ee
then we have $c_{p-1} = 1$, and hence
\be
\sum_{k=0}^{p-1} c_k k! = \left( \sum_{k=0}^{p-2} c_k k! \right) + (p-1)! \not\equiv 0 \pmod{p}.
\ee
Therefore, we have
\be
\sum_{k=0}^{p-1} c_k k! \not\equiv 0 \pmod{p}.
\ee
in both cases.
\end{proof}

The first $100$ entries $(c_k)_{k=0}^{99}$ of the example in the proof are
\be
0, -1, 1, -1, -1, 1, -1, 1, \ldots, 1,
\ee
where the omitted $91$ entries are all $1$, and the heuristic estimation implies the expectation that $-1$ very sparsely appears infinite times, although both of whether $1$ appears infinite times and whether $-1$ appears infinite times are open.

%% file: GroupingMethod.tex
\section{Grouping Method}
\label{Grouping Method}

In this section, we explain the grouping method to effectively search a solution $p \in \bP$ of the congruence relation in \S \ref{Problem Statement}. Before that, we optimise constant factors in computational complexity. There are many ways to optimise constant factors of the brute-force method to search a solution. For example, reducing arithmetic operations by interpretation to a simple recursive formula is considered in \cite{Mij90} and \cite{AT16} for specific settings. The same strategy works in our generalised setting:

\begin{prp}
\label{recursive 1}
Let $n \in \N_{> 0}$. We define a sequence $(a_i)_{i=0}^{n-1} \in \Z^n$ by the following recursive relation:
\be
a_i =
\left\{
\begin{array}{ll}
c_{n-1} & (i = 0) \\
(n-i) a_{i-1} + c_{n-i-1} & (i > 0)
\end{array}
\right.
\ee
Then the equality
\be
\sum_{k=0}^{n-1} c_k k! = a_{n-1}
\ee
holds.
\end{prp}

\begin{proof}
By induction, we have
\be
a_i = \sum_{k=n-1-i}^{n-1} c_k \prod_{j=n-i}^{k} j
\ee
for any $i \in \N_{< n}$. The assertion immediately follows from the equality applied to $i = n$.
\end{proof}

\begin{crl}
\label{recursive 2}
Let $(n,d) \in \N \times \N_{> 0}$. We define a sequence $(a_i)_{i=0}^{n-1} \in \Z^n$ by the following recursive relation:
\be
a_i =
\left\{
\begin{array}{ll}
c_n & (i = 0) \\
\left( \prod_{j=2+d(n-i)}^{1+d(n+1-i)} j \right) a_{i-1} + c_{n-i} & (i > 0)
\end{array}
\right.
\ee
Then the equality
\be
\sum_{k=0}^{n} c_k (1+dk)! = a_n
\ee
holds.
\end{crl}

\begin{proof}
The assertion follows from Proposition \ref{recursive 1} applied to the sequence $(c'_k)_{k=0}^{\infty}$ given as
\be
c'_k \coloneqq
\left\{
\begin{array}{ll}
c_{k/d} & (k \equiv 0 \pmod{d}) \\
0 & (k \not\equiv 0 \pmod{d})
\end{array}
\right..
\ee
\end{proof}

Let $n \in \N$ be a fixed upper bound of prime numbers which we search. Here are pseudocodes for the brute-forth methods to search a solution not greater than $n$ using the recursive relations:

\begin{figure}[H]
\begin{algorithm}[H]
\caption{Brute-forth method to search a solution $p \in \bP_{\leq n}$ of $\sum_{k=0}^{p-1} c_k k! \equiv 0 \pmod{p}$ using Proposition \ref{recursive 1}}
\label{brute-forth 1}
\begin{algorithmic}[1]
\Function {BruteForth}{$n$}
	\State $(p_i)_{i=0}^{\pi(n)-1} \gets$ the enumeration of prime numbers not greater than $n$
	\For {$i \in \N_{< \pi(n)}$ in the usual order}
		\State $a \gets c_{p_i - 1}$
		\For {$k \in \N_{< p_i} \setminus \ens{0}$ in the reverse order}
			\State $a \gets (k a + c_{k - 1}) \bmod p_i$
		\EndFor
		\If {$a = 0$}
			\State \Return $p_i$
		\EndIf
	\EndFor
	\State \Return $-1$
\EndFunction
\end{algorithmic}
\end{algorithm}
\end{figure}

\Call{BruteForth}{$n$} runs in
\be
\Theta \left( n \pi(n) (t_{\r{m}}(n) + t_{\r{r}}(n)) \right)
\ee
time complexity and
\be
\Theta(n s_{\r{m}}(n) + s_{\r{r}}(n))
\ee
space complexity in cases with no solution not greater than $n \in \N_{> 0}$.

\begin{figure}[H]
\begin{algorithm}[H]
\caption{Brute-forth method to search a solution $p \in \bP_{\leq n}$ of $\sum_{k=0}^{\lceil (p_i-1)/d \rceil - 1} c_k (1+dk)! \equiv 0 \pmod{p}$ using corollary \ref{recursive 2}}
\label{brute-forth 2}
\begin{algorithmic}[1]
\Function {SkipBruteForth}{$n,d$}
	\State $(p_i)_{i=0}^{\pi(n)-1} \gets$ the enumeration of prime numbers not greater than $n$
	\For {$i \in \N_{< \pi(n)}$ in the usual order}
		\State $B' \gets \lceil (p_i-1)/d \rceil$
		\State $a \gets c_{B' - 1}$
		\For {$k \in \N_{< B'} \setminus \ens{0}$ in the reverse order}
			\State $t \gets 1$
			\For {$k' \in \N \cap (1+d(k-1),1+dk]$ in the usual order}
				\State $t \gets t \times k' \bmod p_i$
			\EndFor
			\State $a \gets (ta + c_{k - 1}) \bmod p_i$ \Comment{$t = (1+dk)!/(1+d(k-1))! \bmod {p_i}$}
		\EndFor
		\If {$a = 0$}
			\State \Return $p_i$
		\EndIf
	\EndFor
	\State \Return $-1$
\EndFunction
\end{algorithmic}
\end{algorithm}
\end{figure}

\Call{SkipBruteForth}{$n,d$} runs in
\be
\Theta \left( \pi(n) (t_{\r{r}}(d) + n (t_{\r{m}}(n) + t_{\r{r}}(n))) \right)
\ee
time complexity and
\be
\Theta(n s_{\r{m}}(n) + s_{\r{r}}(\max \ens{n,d})))
\ee
space complexity in cases with no solution not greater than $n \in \N_{> 0}$.

\vs
Because of the motivation explained in \S \ref{Introduction}, we do not consider constant-factor optimisations related to parallel computation. For example, Montgomery reduction is a standard method to effectively compute $x \bmod y$ for varying $x \in \Z$ and fixed $y \in \N_{> 0}$ in a way compatible with parallel computation, but does not perform well in non-parallel computation. Instead, we focus on reducing arithmetic operations by grouping prime numbers.

\vs
Let $B \in \N_{> 0}$ be a fixed constant for the size of groups of prime numbers. The grouping method is based on the following parallel criterion for $B$ prime numbers:

\begin{prp}
\label{chinese remainder theorem}
For any strictly increasing sequence $(p_i)_{i=0}^{B-1} \in \bP^B$, the following are equivalent:
\bi
\item[(1)] There exists some $i \in \N_{< B}$ such that the congruence relation
\be
\sum_{k=0}^{p_i-1} c_k k! \equiv 0 \pmod{p_i}
\ee
holds.
\item[(2)] The congruence relation
\be
\sum_{k=0}^{p_{B-1}-1} c_k k! \equiv 0 \pmod{q}
\ee
holds, where $q$ denotes $\prod_{i=0}^{B-1} p_i$.
\ei
\end{prp}

\begin{proof}
Since $(p_i)_{i=0}^{B-1}$ is strictly increasing, its maximum is $p_{B-1}$ and it is pairwise coprime. Therefore, we have
\be
\sum_{k=0}^{p_{B-1}-1} c_k k! \equiv \sum_{k=0}^{p_i-1} c_k k! \pmod{p_i}
\ee
for any $i \in \N_{< B}$ and hence the assertion follows from Chinese remainder theorem.
\end{proof}

Proposition \ref{chinese remainder theorem} helps us to reduce total amount of arithmetic operations in the repetition of the decision processes (Algorithm \ref{brute-forth 1} line 4--10 and Algorithm \ref{brute-forth 2} line 4--15) with trade-off against the computational complexity of big integer arithmetic. In addition, we reduce residue operations by lazy evaluation of modular arithmetic also by big integer arithmetic. Here are pseudocodes for the grouping methods to search a solution not greater than $n$:

\begin{figure}[H]
\begin{algorithm}[H]
\caption{Grouping method to search a solution $p \in \bP_{\leq n}$ of $\sum_{k=0}^{p-1} c_k k! \equiv 0 \pmod{p}$ using Proposition \ref{recursive 1}}
\label{grouping 1}
\begin{algorithmic}[1]
\Function {Grouping}{$n,B$}
	\State $(p_i)_{i=0}^{\pi(n)-1} \gets$ the enumeration of prime numbers not greater than $n$
	\For {$b_0 \in \N_{< \lceil \pi(n)/B \rceil}$ in the usual order} \Comment{$b_0 < \lceil \pi(n)/B \rceil \Leftrightarrow Bb_0 < \pi(n)$}
		\State $i_0 \gets Bb_0$
		\State $i_1 \gets \min \ens{B(b_0+1),\pi(n)}$
		\State $q \gets 1$
		\For {$i \in \N \cap [i_0,i_1)$ in the usual order}
			\State $q \gets qp_i$
		\EndFor
		\State $B' \gets \lceil (p_{i_1-1}-1)/B \rceil$
		\State $a \gets c_{BB'}$
		\For {$b_1 \in \N_{< B'}$ in the reverse order} \Comment{$b_1 < B' \Leftrightarrow 1+Bb_1 < p_{i_1-1}$}
			\For {$k \in \N \cap [1+Bb_1,1+B(b_1+1))$ in the reverse order}
				\State $a \gets k a + c_{k - 1}$
			\EndFor
			\State $a \gets a \bmod q$
		\EndFor
		\For {$i \in \N \cap [i_0,i_1)$ in the usual order}
			\If {$a \bmod p_i = 0$}
				\State \Return $p_i$
			\EndIf
		\EndFor
	\EndFor
	\State \Return $-1$
\EndFunction
\end{algorithmic}
\end{algorithm}
\end{figure}

\Call{Grouping}{$n,B$} runs in
\be
\Theta \left( (n \log \log n) t_{\r{m}}(n) + \pi(n) \left( \left( 1 + \frac{n}{B} \right) t_{\r{m}} \left( n^B \right) + \frac{1}{B} t_{\r{r}} \left( \max \ens{n,B} \right) + \frac{n}{B^2} t_{\r{r}} \left( n^B \right) \right) \right)
\ee
time complexity and
\be
\Theta \left( n s_{\r{m}}(n) + s_{\r{m}} \left( n^B \right) + s_{\r{r}} \left( n^B \right) \right)
\ee
space complexity in cases with no solution not greater than $n \in \N_{> 0}$.

\begin{figure}[H]
\begin{algorithm}[H]
\caption{Grouping method to search a solution $p \in \bP_{\leq n}$ of $\sum_{k=0}^{\lceil (p-1)/d \rceil - 1} c_k (1+dk)! \equiv 0 \pmod{p}$ using Corollary \ref{recursive 2} under $d < B$}
\label{grouping 2}
\begin{algorithmic}[1]
\Function {SkipGrouping}{$n,d,B$}
	\State $(p_i)_{i=0}^{\pi(n)-1} \gets$ the enumeration of prime numbers not greater than $n$
	\State $B \gets B - (B \bmod d)$
	\For {$b_0 \in \N_{< \lceil \pi(n)/B \rceil}$ in the usual order} \Comment{$b_0 < \lceil \pi(n)/B \rceil \Leftrightarrow Bb_0 < \pi(n)$}
		\State $i_0 \gets Bb_0$
		\State $i_1 \gets \min \ens{B(b_0+1),\pi(n)}$
		\State $q \gets 1$
		\For {$i \in \N \cap [i_0,i_1)$ in the usual order}
			\State $q \gets qp_i$
		\EndFor
		\State $B' \gets \lceil (p_{i_1-1}-1)/B \rceil$
		\State $a \gets c_{(B/d)(B'-1)}$
		\For {$b_1 \in \N_{< B'} \setminus \ens{0}$ in the reverse order} \Comment{$b_1 < B' \Leftrightarrow 1+Bb_1 < p_{i_1-1}$}
			\For {$k \in \N \cap ((B/d)(b_1-1),(B/d)b_1]$ in the reverse order}
				\State $t \gets 1$
				\For {$k' \in \N \cap (1+d(k-1),1+dk]$ in the usual order}
					\State $t \gets t \times k'$
				\EndFor
				\State $a \gets ta + c_{k-1}$ \Comment{$t = (1+dk)!/(1+d(k-1))!$}
			\EndFor
			\State $a \gets a \bmod q$
		\EndFor
		\For {$i \in \N \cap [i_0,i_1)$ in the usual order}
			\If {$a \bmod p_i = 0$}
				\State \Return $p_i$
			\EndIf
		\EndFor
	\EndFor
	\State \Return $-1$
\EndFunction
\end{algorithmic}
\end{algorithm}
\end{figure}

\Call{SkipGrouping}{$n,d,B$} runs in
\be
\Theta \left( (n \log \log n) t_{\r{m}}(n) + t_{\r{r}}(d) + \pi(n) \left( \left( 1 + \frac{n}{B} \right) t_{\r{m}} \left( n^B \right) + \frac{1}{B} t_{\r{r}} \left( \max \ens{n,B} \right) + \frac{n}{B^2} t_{\r{r}} \left( n^B \right) \right) \right)
\ee
time complexity and
\be
\Theta \left( n s_{\r{m}}(n) + s_{\r{r}}(d) + s_{\r{m}} \left( n^B \right) + s_{\r{r}} \left( n^B \right) \right)
\ee
space complexity in cases with no solution not greater than $n \in \N_{> 0}$.

%% file: Application.tex
\section{Application}
\label{Application}

We searched a solution for the case
\be
c_k =
\left\{
\begin{array}{ll}
0 & (k = 0) \\
(-1)^{k-1} & (k > 0)
\end{array}
\right.,
\ee
i.e.\ a solution $p \in \bP$ of the congruence relation
\be
\sum_{k=1}^{p-1} (-1)^{k-1} k! \equiv 0 \pmod{p},
\ee
by executing \Call{Grouping}{$3 \times 10^7,10^5$}. We note that $\check{\r{Z}}$.\ Mijajlovi\'c spent about $130$ hours to find the least solution $p = 3612703$ using the brute-force method in \cite{Mij99}, but we spent only $20$ minutes with a standard laptop PC. We verified that there is no other solution below $3 \times 10^7$, and it is still open that there is a solution greater than the least solution $3612703$.

\vs
We also searched a solution for the case
\be
c_k =
\left\{
\begin{array}{ll}
1 & (k \equiv 1 \pmod{d}) \\
0 & (k \not\equiv 1 \pmod{d}) \\
\end{array}
\right.,
\ee
i.e.\ a solution $p \in \bP$ of the congruence relation
\be
\sum_{k=0}^{\lceil (p-1)/d \rceil - 1} (1+kd)! \equiv 0 \pmod{p},
\ee
for $d \in \N_{\leq 25} \setminus \ens{0}$ by executing \Call{SkipGrouping}{$2 \times 10^7,d,10^5$}. Here is the table of the obtained least solution $p$:

\begin{table}[H]
\caption{The least solution $p$ for $d$\label{searching result}}
\begin{center}
\begin{minipage}{0.15\textwidth}
\centering
\begin{tabular}{|r|c|}
\hline
$d$ & $p$ \\
\hline \hline
1 & 3 \\
2 & 107 \\
3 & 5 \\
4 & 6323 \\
5 & 7 \\
\hline
\end{tabular}
\end{minipage}
\begin{minipage}{0.15\textwidth}
\centering
\begin{tabular}{|r|c|}
\hline
$d$ & $p$ \\
\hline \hline
6 & 61 \\
7 & 19 \\
8 & ? \\
9 & 11 \\
10 & 4093 \\
\hline
\end{tabular}
\end{minipage}
\begin{minipage}{0.15\textwidth}
\centering
\begin{tabular}{|r|c|}
\hline
$d$ & $p$ \\
\hline \hline
11 & 13 \\
12 & 257 \\
13 & 23 \\
14 & 2549 \\
15 & 17 \\
\hline
\end{tabular}
\end{minipage}
\begin{minipage}{0.20\textwidth}
\centering
\begin{tabular}{|r|c|}
\hline
$d$ & $p$ \\
\hline \hline
16 & 32261 \\
17 & 19 \\
18 & ? \\
19 & 3977273 \\
20 & ? \\
\hline
\end{tabular}
\end{minipage}
\begin{minipage}{0.15\textwidth}
\centering
\begin{tabular}{|r|c|}
\hline
$d$ & $p$ \\
\hline \hline
21 & 23 \\
22 & 61 \\
23 & ? \\
24 & 1087 \\
25 & ? \\
\hline
\end{tabular}
\end{minipage}
\end{center}
\end{table}

We verified that there is no solution $p$ for $d \in \ens{8,18,20,23,25}$ below $2 \times 10^7$, and it is open whether there is a solution for $d$ or not. We finish the paper by proving Theorem \ref{main}:

\begin{proof}[Proof of Theorem \ref{main}]
The assertion immediately follows from Proposition \ref{finiteness} applied to the case
\be
c_k =
\left\{
\begin{array}{ll}
1 & (k \equiv 1 \pmod{d}) \\
0 & (k \not\equiv 1 \pmod{d}) \\
\end{array}
\right.
\ee
with $d \in \N_{\leq 25} \setminus \ens{0,8,18,20,23,25}$ by Table \ref{searching result}.
\end{proof}

%% file: References.tex
% \newpage
\vspace{0.3in}
\addcontentsline{toc}{section}{Acknowledgements}
\noindent {\Large \bf Acknowledgements}
\vspace{0.2in}

\noindent
I thank all people who helped me to learn mathematics and programming. I also thank my family.